\documentclass[]{interact}

\usepackage{epstopdf}
\usepackage[caption=false]{subfig}

\usepackage[numbers,sort&compress]{natbib}

\usepackage{color}
\usepackage{xcolor}
\usepackage[makeroom]{cancel}
\usepackage{soul}
\definecolor{mypink1}{rgb}{1.00,0.10,1.00}

\bibpunct[, ]{[}{]}{,}{n}{,}{,}
\makeatletter
\def\NAT@def@citea{\def\@citea{\NAT@separator}}
\makeatother

\theoremstyle{plain}
\newtheorem{theorem}{Theorem}[section]

\newtheorem{corollary}[theorem]{Corollary}
\newtheorem{proposition}[theorem]{Proposition}
\newtheorem{fact}[theorem]{Fact}

\theoremstyle{definition}
\newtheorem{definition}[theorem]{Definition}
\newtheorem{example}[theorem]{Example}
\newtheorem{lemma}[theorem]{Lemma}

\theoremstyle{remark}
\newtheorem{remark}{Remark}

\newtheorem{problem}{Problem}

\newcommand\cC{{{\mathcal C}}}

\newcommand\cK{{{\mathcal K}}}

\newcommand\R{\mathbb R}

\newcommand\cero{0_X}

\newcommand\normx{\|x \|}
\newcommand\normf{\| f \|_*}

\begin{document}


\title{Existence and Density Theorems of  Henig Global Proper Efficient Points}

\author{
\name{Fernando García-Castaño\thanks{Fernando García-Castaño  https://orcid.org/0000-0002-8352-8235    Email:fernando.gc@ua.es},  Miguel Ángel Melguizo-Padial\thanks{Miguel Ángel Melguizo-Padial  https://orcid.org/0000-0003-0303-791X   Email:ma.mp@ua.es}}
\affil{Department of Mathematics, University of Alicante,  03080,  San Vicente del Raspeig,  Alicante,  Spain}
}

\maketitle

\begin{abstract}
In this work, we provide some novel results that establish both the existence of Henig global proper efficient points and their density in the efficient set for vector optimization problems in arbitrary normed spaces. Our results do not require the assumption of convexity, and in certain cases,  can be applied to unbounded sets. However, it is important to note that a weak compactness condition on the set (or on a section of it) and a separation property between the order cone and its conical neighborhoods remains necessary. The weak compactness condition ensures that certain convergence properties hold. The separation property enables the interpolation of a family of Bishop-Phelps cones between the order cone and each of its conic neighborhoods. This interpolation, combined with the proper handling of two distinct types of conic neighborhoods, plays a crucial role in the proofs of our results, which include as a particular case other results that have already been established under more restrictive conditions.
\end{abstract}


\begin{amscode}
 90C29;90C30;49J27;46N10
\end{amscode}

\section{Introduction}
Proper efficient points  were introduced to eliminate those efficient points that cannot be satisfactorily characterized by a scalar problem. This notion  was first introduced  in the  pioneering paper \cite{Geoffrion1968}, where  properness was defined in $\mathbb{R}^n$ with respect to the nonnegative octant. Later on,  many efforts were accomplished to refine this concept  in order to obtain more characterizations of these points by means of positive functionals. Most important refinements of efficient solutions are due to different authors as Benson, Geoffrion, Hartley, Henig, Borwein, and some others,  giving rise to a wide range of notions of proper efficiency.  In this paper, we focus  on the notion of Henig global proper efficiency. This type of proper efficiency was introduced and characterized  in \cite{Henig1982} by using a separation property given by a cone. This  characterization became useful  to unify and generalize the notions of proper efficiency previously established. Subsequently, Henig global proper efficiency has been generalized to infinite-dimensional spaces and it has been object of many investigations (see \cite{Zheng1997,Makarov1999,Gong2001,Gong2005,Qiu2010,Kasimbeyli2010,
Newhall2014,Kasimbeyli2015,Chen2018,Zhou2018,Kasimbeyli2019,Tung2020, Hai2021,GARCIACASTANO2022} and the references therein). One of the desired properties for proper efficient points is their density within the set of all efficient points of the problem. The first density result was established by
Arrow,  Barankin and Blackwell in \cite{ArrowBarankinBlackwell1953} showing that the set of positive proper efficient points of a compact subset of $\mathbb{R}^n$ is dense in the set of efficient points. This theorem has been widely generalized, see for example
\cite{Borwein1977,Hartley1978,Borwein1980,Henig1982,Jahn1988,Gong1995} and the references therein.  Density results for Henig proper efficient points in nonconvex sets can be found in \cite{Fu1996,Makarov1996,Bednarczuk1999,Gopfert2004,Kasimbeyli2016}.

In this paper, we utilize the separation property introduced in \cite{GARCIACASTANO2022} between the ordering cone and its conical neighbourhoods to establish sufficient conditions for the existence of Henig global proper efficient points in general normed spaces.  As is customary in existence results, we employ certain compactness conditions, specifically weak compactness. In particular, we present a theorem that requires the weak compactness of a section of the set, where the section is defined by a conical neighborhood. Additionally, we introduce another existence result that requires the weak compactness of the entire considered set. To establish this latter result, we combine the separation property with two types of conical neighborhoods, allowing us to leverage the convexity of the so-called Henig dilating cones. Furthermore, we extend our investigation to determine new nonconvex density results for Henig global proper minimal points in general normed spaces.  In the density results, we once again apply both the separation property and the weak compactness conditions from the existence results. However, in this case, we also require the ordering cone to be closed. Once again, we combine the use of the two types of conic neighborhoods to obtain the result that assumes the entire considered set is weakly compact.  Remarkably, our results do not rely on convexity and can be applied to unbounded sets in some cases. This fact highlights the generality and flexibility of our results, as they can be applied to a wide range of sets and problems beyond the traditional convex setting. By relaxing the convexity assumption, we open up possibilities for exploring optimization problems in non-convex domains, where traditional convex methods may not be applicable. Additionally, the relaxation of boundedness assumptions allows for the consideration of unbounded sets, which can arise in various practical scenarios. In addition, our results also enhance and expand upon the main results presented in \cite{Kasimbeyli2015,Kasimbeyli2016}. While those previous works assumed the reflexivity of the underlying space, our results do not require this additional assumption. This relaxation of reflexivity allows for the application of our results in a wider range of normed spaces, further increasing the practical relevance and generality of the obtained conclusions.

The paper is organized as follows. In Section \ref{Notation}, we establish the general terminology and notation to be used throughout the work.  In Section~\ref{sec:tmas_separacion_conos}, we present novel results concerning the separation property initially introduced in \cite{GARCIACASTANO2022}. We explore two distinct types of dilating cones and provide a characterization of the separation property for each one. Our results demonstrate the possibility of interpolating Bishop-Phelps cones between the ordering cone and the corresponding conic dilation (or conic neighborhood). Specifically, we establish Theorems \ref{thm:tma_separacion_epsilon_conic_neighbourhood} and \ref{thm:tma_separacion_Henig_dilating_cone}. The use of $\epsilon$-conic neighborhoods allows for a more simplified statement of the main results in Section \ref{sec:exist_henig}. Additionally, the utilization of Henig dilating cones, leveraging their convexity, allows us to derive results by assuming the weak compactness of the entire set involved in the vector optimization problem. In Remark \ref{rem:resumen_existencia_conos}, we summarize several properties that highlight the interplay between these two types of dilating cones and the separation property. This remark proves particularly valuable in the subsequent section. In Section \ref{sec:exist_henig}, we present the main results of the paper. Theorem \ref{thm:Existence_He_Hartley_compact_nuevo} establishes a sufficient condition for the existence of Henig global proper efficient points, assuming the weak compactness of a certain section of the set given by a conic neighbourhood. Theorem~\ref{thm:Existence_GHe_con_Henig_dilating_cones_no_reflexivos} is a variant of Theorem~\ref{thm:Existence_He_Hartley_compact_nuevo} specifically tailored for Henig dilating cones, leading to Theorem \ref{thm:Existence_He_Cjto_compactbis}, which guarantees the existence of Henig global proper minimal points in weakly compact sets. The latter part of this section is devoted to establishing density results akin to the theorem by Arrow, Barankin, and Blackwell. Theorem \ref{thm:arrow_barankin_blackwell_local_approximation} presents our first density result, providing a local approximating theorem under the assumption that an arbitrary section containing the efficient point is weakly compact. Finally, Theorem \ref{thm:arrow_barankin_blackwell_A_weak_compact} establishes the density of Henig global proper efficient points in the set of minimal points for weakly compact sets.

\section{Notation and preliminaries}\label{Notation}
Throughout the paper, we use the following notation in the context of a normed space $X$. Let $\|\cdot\|$ denote the norm on $X$, $X^*$ the dual space of $X$, $\|\cdot\|_*$ the norm on $X^*$, and $0_X$ the origin of $X$. By $B_X$ (resp. $B_X^\circ$), we denote the closed (resp. open) unit ball of $X$, and by $S_X$ we denote the unit sphere, i.e., $B_X:=\{x \in X \colon \|x\| \leq 1\}$, $B_X^\circ:=\{x \in X \colon \|x\| < 1\}$, and $S_X:=\{x \in X \colon \|x\| = 1\}$, respectively. Given $x \in X$ and $r > 0$, we define $B(x,r):=\{y \in X \colon \|x-y\| \leq r\}$ and $B^\circ(x,r):=\{y \in X \colon \|x-y\| < r\}$. For a subset $A \subset X$, we denote by $\overline{A}$ (resp. bd($A$), int($A$), co($A$), $\overline{\text{co}}(A)$) the closure (resp. the boundary, the interior, the convex hull, the closure of the convex hull) of $A$. By $\mathbb{R}_+$ (resp. $\mathbb{R}_{++}$, $\mathbb{N}$), we denote the set of nonnegative real numbers (resp. strictly positive real numbers, natural numbers). A subset $\mathcal{C} \subset X$ is called a cone if $\lambda x \in \mathcal{C}$ for every $\lambda \geq 0$ and $x \in \mathcal{C}$. Let $\mathcal{C} \subset X$ be a cone: $\mathcal{C}$ is said to be nontrivial if $\{\cero\} \subsetneq \mathcal{C} \subsetneq X$, $\mathcal{C}$ is said to be convex if $\cC+\cC=\cC$, and $\mathcal{C}$ is said to be pointed if $(-\mathcal{C}) \cap \mathcal{C} = \{\cero\}$. All cones in this manuscript are assumed to be nontrivial unless stated otherwise.

Given a subset $A \subset X$, we define the cone generated by $A$ as cone$(A):=\{\lambda s \colon \lambda \geq 0,\, s \in A\}$, and $\overline{\text{cone}}(A)$ stands for the closure of cone$(A)$. A nonempty convex subset $B$ of a convex cone $\mathcal{C}$ is said to be a base for $\mathcal{C}$ if $\cero \notin \overline{B}$, and for every $x \in \mathcal{C} \setminus \{\cero\}$, there exist unique $\lambda_x > 0$ and $b_x \in B$ such that $x = \lambda_x b_x$. Given a cone $\mathcal{C} \subset X$, its dual cone is defined by $\mathcal{C}^*:=\{f \in X^* \colon f(x) \geq 0,\, \forall x \in \mathcal{C}\}$, and the quasi-relative interior of $\mathcal{C}^*$ by $\mathcal{C}^\#:=\{f \in X^* \colon f(x) > 0,\, \forall x \in \mathcal{C},\, x \neq \cero\}$. The set $\mathcal{C}^\#$ is also called the set of strictly positive functionals. A convex cone $\mathcal{C} \subset X$ has a base if and only if $\mathcal{C}^\# \neq \emptyset$, and the latter implies that $\mathcal{C}$ is pointed. In particular, for every $f \in \mathcal{C}^\#$, the set $B:=\{x \in \mathcal{C} \colon f(x) = 1\}$ is a base for $\mathcal{C}$. A convex cone $\mathcal{C}$ is said to have a bounded base if there exists a base $B$ for $\mathcal{C}$ that is a bounded subset of $X$.  We refer the reader to \cite{Jahn2004} for the previously mentioned facts.  In \cite{Kasimbeyli2010}, the augmented dual cones were defined as $\mathcal{C}^{a*}:=\{(f,\alpha) \in \mathcal{C}^\# \times \mathbb{R}_+ \colon f(x) - \alpha\|x\| \geq 0,\, \forall x \in \mathcal{C}\}$ and $\mathcal{C}^{a\#}:=\{(f,\alpha) \in \mathcal{C}^\# \times \mathbb{R}_+ \colon f(x) - \alpha\|x\| > 0,\, \forall x \in \mathcal{C},\, x \neq \cero\}$. In \cite{GARCIACASTANO2022}, the following augmented dual cones were also introduced: $\mathcal{C}^{a*}_+:=\{(f,\alpha) \in \mathcal{C}^\# \times \mathbb{R}_{++} \colon f(x) - \alpha\|x\| \geq 0,\, \forall x \in \mathcal{C}\}$ and $\mathcal{C}^{a\#}_+:=\{(f,\alpha) \in \mathcal{C}^\# \times \mathbb{R}_{++} \colon f(x) - \alpha\|x\| > 0,\, \forall x \in \mathcal{C},\, x \neq \cero\}$.  In addition, the sublevel set corresponding to $f \in X^*$ and $\alpha \geq 0$ is defined by $S(f,\alpha):=\{x \in X \colon f(x) + \alpha\|x\| \leq 0\}$. 

On the other hand,  any pointed convex cone $\mathcal{C} \subset X$ provides a partial order on $X$, denoted by $\leq$, through the relationship $x \leq y \Leftrightarrow y - x \in \mathcal{C}$. In this setting, we say that $X$ is a partially ordered normed space, and $\mathcal{C}$ is the ordering cone. Given a subset $A \subset X$, we say that $x_0 \in A$ is an efficient (or Pareto minimal) point of $A$, written $x_0 \in \text{Min}(A,\mathcal{C})$, if $(A - x_0) \cap (-\mathcal{C}) = \{\cero\}$. On the other hand, we say that $x_0 \in A$ is a Henig global proper efficient point of $A$, written $x_0 \in \text{GHe}(A,\mathcal{C})$, if $x_0 \in \text{Min}(A,\mathcal{K})$ for some convex cone $\mathcal{K}$ such that $\mathcal{C} \setminus \{\cero\} \subset \text{int}(\mathcal{K})$. In the last section of this paper, we provide conditions under which the set $\text{GHe}(A,\mathcal{C})$ is non-empty and dense in $\text{Min}(A,\mathcal{C})$.
\section{Separation theorems and dilating cones}\label{sec:tmas_separacion_conos}
The main objective of this section is to establish  new theorems that characterize when the ordering cone and its conical neighborhoods satisfy the separation property introduced in \cite{GARCIACASTANO2022}. We begin by introducing the separation property and some technical results before defining the two types of conical neighborhoods (or dilation cones) that we will use in this work. In what follows, for a normed space $X$ and two subsets $A$, $B\subset X$, we will write $A-B:=\{x-y\colon x\in A,\, y \in B\}$. In addition, to simplify notation, we will write $A_0$ instead of $A\cup \{\cero\}$.
\begin{definition}\label{def:estrict_separation}
Let $X$ be a normed space and $\cC$, $\cK$ cones on $X$. We say that the pair of cones $(\cC,\cK)$ has the strict separation property (SSP for short) if $\cero\not \in \overline{{\mbox{co}}(\cC\cap S_X)-\mbox{co}((\mbox{bd}(\cK)\cap S_X)_0)}$.
\end{definition}

The SSP depends on the pair of cones $(\cC,\cK)$ and the geometry unit sphere, $S_X$. The following example shows that the choice of the norm can be crucial for the SSP.

\begin{example}
Let us take $X=\mathbb{R}^2$ and  the  cones 
$\cC:=\{(x_1,x_2)\in \mathbb{R}^2: x_2 \geq \sqrt{3} | x_1|\}$ and 
$\cK:=\{(x_1,x_2)\in \mathbb{R}^2: x_2 \geq | x_1|\}$. Then $(\cC,\cK)$ has the SSP in $(\R^2,\|\cdot \|_2)$ but it does not enjoy the SSP in $(\R^2,\|\cdot \|_{\infty})$.
\end{example}
\begin{proof}
Now, first, considering the euclidean norm $\| (x_1,x_2) \|_2: =\sqrt{x_1^2+x_2^2}$ for any $(x_1,x_2)\in \mathbb{R}^2$. Then, we have that 
$\mbox{co}(\cC \cap S_X)= \{(x_1,x_2)\in \mathbb{R}^2: \sqrt{x_1^2+x_2^2}\leq 1, \,  x_2\geq \frac{\sqrt{3} }{2}\}$ and 
$ \mbox{co}((\mbox{bd}(\cK)\cap S_X)_0)=\cK \cap \{(x_1,x_2)\in \mathbb{R}^2: x_2\leq \frac{\sqrt{2} }{2} \}$. Consequently, $\cero\not \in \overline{{\mbox{co}}(\cC\cap S_X)-\mbox{co}((\mbox{bd}(\cK)\cap S_X)_0)}$.   Consider now  the norm $\| (x_1,x_2) \|_{\infty}: =\max\{|x_1|,|x_2|\}$ for any $(x_1,x_2)\in \mathbb{R}^2$. We have that 
$\mbox{co}(\cC \cap S_X)= \{(x_1,x_2)\in \mathbb{R}^2: x_1 \in [-\frac{1}{\sqrt{3}},\frac{1}{\sqrt{3}}], x_2=1   \}$ and $ \mbox{co}((\mbox{bd}(\cK)\cap S_X)_0)=\cK \cap \{(x_1,x_2)\in \mathbb{R}^2: x_2\leq 1 \}$. Consequently,  $\mbox{co}(\cC \cap S_X)\subset  \mbox{co}((\mbox{bd}(\cK)\cap S_X)_0),$ and then
$\cero  \in \overline{{\mbox{co}}(\cC\cap S_X)-\mbox{co}((\mbox{bd}(\cK)\cap S_X)_0)}.$

\end{proof}

Here we recall that given  $f \in  X^*$ and $0<\alpha <  \| f \|_* $, the  closed  convex cone
$$C(f,\alpha):=\{x \in X: f(x)- \alpha \| x\| \geq 0 \}  $$
is said to be a Bishop-Phelps cone. These kind of cones verify the following lemma.

\begin{lemma}\label{LemaBP}
Let $(X, \| \, \|)$ be a normed space,  $f \in S_{X^*}$, and   $0<\alpha <  1  $. The following statements hold. 
\begin{itemize}
\item[$(a)$] $ \mbox{co} (  C(f,\alpha) \cap S_X)\subset \{ f \geq \alpha \}$. 
\item[$(b)$] $ \mbox{co} ( (\mbox{bd}(C(f,\alpha))\cap S_X)_0) \subset \{ f \leq \alpha \}$.
\end{itemize}
\end{lemma} 
\begin{proof}
$(a)$ Fix some $x \in \mbox{co} (  C(f,\alpha) \cap S_X)$. 
Consider the case $x \in   C(f,\alpha) \cap S_X$. From $\|x\|=1 $ and $x \in  C(f,\alpha)$, we get that $f(x)\geq \alpha$.  Now, assume that $x= \sum_{i=1}^n \beta_i x_i $ with $0<\beta_i<1$ and $x_i \in  C(f,\alpha) \cap S_X$  for any $1\leq i \leq n $, $\sum_{i=1}^n \beta_i=1$. Then,
$f(x)= \sum_{i=1}^n \beta_i f(x_i) \geq \sum_{i=1}^n \beta_i \alpha= \alpha$.

$(b)$ Take some $x \in \mbox{co} ( (\mbox{bd}(C(f,\alpha))\cap S_X)_0)$. 
First, if $x \in   \mbox{bd}(C(f,\alpha)) \cap S_X$, as $x \in   \mbox{bd}(C(f,\alpha))
\Rightarrow f(x)= \alpha \|x\|$, and since $\|x\|=1 $, we get that $f(x)=\alpha$. 
Assume now that $x= \sum_{i=1}^n \beta_i x_i $ with $0<\beta_i<1$ and $x_i \in \mbox{bd}(C(f,\alpha))\cap S_X  $ or $x_i =0_X$  for any $1\leq i \leq n $, $\sum_{i=1}^n \beta_i=1$.
Then,
$f(x)= \sum_{i=1}^n \beta_i f(x_i) \leq \sum_{i=1}^n \beta_i \alpha= \alpha$.
\end{proof}

In the following example we present two Bishop-Phelps cones having the SSP.

\begin{example} \label{EjempoBPtieneSSP}
Let $f \in S_{X^*}$ and fix two real numbers $0<\alpha_1< \alpha_2 <  1  $. Then $(C(f, \alpha_2),C(f, \alpha_1)) $ has the SSP.
\end{example}
\begin{proof}
Indeed,  by Lemma \ref{LemaBP}, we get that $ \mbox{co} (  C(f,\alpha) \cap S_X)\subset \{ f \geq \alpha_2 \}$  and $\mbox{co} ( (\mbox{bd}(C(f,\alpha))\cap S_X)_0) \subset \{ f \leq \alpha_1 \}$. Now since $\cero\not \in \overline{\{ f \geq \alpha_2 \} - \{ f \leq \alpha_1 \}}$, we finally get that $\cero\not \in \overline{{\mbox{co}}(\cC\cap S_X)-\mbox{co}((\mbox{bd}(\cK)\cap S_X)_0)}$.

\end{proof}

It is clear that $(\cC,\cK)$ possesses the SSP if and only if $(-\cC,-\cK)$ possesses the SSP. This equivalence  will be utilized in the proof of the main results in this section.  Our initial result ensures that a cone in the first coordinate of a cone pair that satisfies the SSP has a bounded base.
\begin{proposition}\label{prop:separacion_implica_base_acotada}  
Let $X$ be a normed space and $\cC$, $\cK$ cones on $X$. If $\cC$ is convex and  $(\cC,\cK)$ has the SSP, then $\cC$ has a bounded base.  
\end{proposition}
To demonstrate the previous result, we will rely on the following two lemmas concerning augmented dual cones. It is worth mentioning that certain statements improve upon technical outcomes presented in \cite{Kasimbeyli2010}.
\begin{lemma}\label{lem:S(f,alpha)_base_acotada}
Let $X$ be a normed space and $\cC\subset X$ a cone. The following statements hold.
\begin{itemize}
\item[(i)]$\cC^{a*}_+ \not = \emptyset$ if and only if $\cC^{a\#}_+ \not = \emptyset$.
\item[(ii)] If $(f,\alpha)\in \cC^{a*}$ (resp. $(f,\alpha)\in \cC^{a\#}$), then $\alpha \leq \| f \|_*$ (resp. $\alpha < \| f \|_*$).
\item[(iii)] If $(f,\alpha)\in \cC^{a*}$ and $\alpha>0$, then $S(f,\alpha)$  is a closed convex cone with a bounded base. In addition, we have $-\cC\subset S(f,\alpha)$ for every $0 \leq \alpha \leq \| f \|_*$. 
\item[(iv)] If $f \in X^*$ and $0<\alpha<\| f \|_*$, then int$(S(f,\alpha))=\{x \in X\colon f(x)+\alpha \normx <0\}\not = \emptyset$.
\end{itemize}
\end{lemma}
\begin{proof}
(i) It is a consequence of the inclusion $\cC^{a\#}_+\subset \cC^{a*}_+$ and \cite[Lemma 3.2, (ii)]{Kasimbeyli2010}.

(ii) Let $(f,\alpha)\in \cC^{a*}$. Then $f(x)\geq \alpha$, $\forall x \in \cC\cap S_X$. Now, by definition of dual norm as a supremum over $S_X$, we have $\alpha \leq \| f \|_*$. The case $(f,\alpha)\in \cC^{a\#}$ is analogous.

(iii) It is easy to check that $S(f,\alpha)$ is a cone.  For the proof of the inclusion $-\cC \subset S(f,\alpha)$, we refer the reader to \cite[Lemma 3.2]{Kasimbeyli2010}. Finally, we will check that $S(f,\alpha)$ is convex and it has a bounded base. We fix $x_1$, $x_2 \in S(f,\alpha)$ and $\lambda \in (0,1)$. Then $f(\lambda x_1+(1-\lambda)x_2)+\alpha\| \lambda x_1+(1-\lambda)x_2\| \leq \lambda[f(x_1)+\alpha\| x_1\|]+(1-\lambda)[f(x_2)+\alpha\| x_2\|]\leq 0$. In order to check that $S(f,\alpha)$ has a bounded base, we define $g:=-f\in X^*$ and consider $x \in S(f,\alpha)$, $x \not = \cero$. Then $f(x)+\alpha\| x \| \leq 0$, which implies that $g(\frac{x}{\normx})\geq \alpha >0$. Then $g \in S(f,\alpha)^{\#}$. As a consequence, $B:=\{x\in  S(f,\alpha)\colon g(x)=1\}$ is a base for $S(f,\alpha)$. The base $B$ is bounded because $\normx\leq \frac{1}{\alpha}$ for every $x \in B$. Indeed, if $x \in B$, then $0 \geq f(x)+\alpha \normx=-1+\alpha \normx \Rightarrow \normx\leq \frac{1}{\alpha}$.

(iv) We define the function $g:X\rightarrow \mathbb{R}$ by $g(x):=f(x)+\alpha \normx$, $\forall x \in X$. Since $g$ is continuous, the set $H:=\{x \in X\colon f(x)+\alpha \normx <0\}$ is open and  contained in $S(f,\alpha)$. Let us check now that $H\not = \emptyset$. As $\alpha < \normf$, there exists $x\in S_X$ such that $\alpha<f(x)\leq \normf$. Then $f(-x)<-\alpha$, which implies $f(-x)+\alpha \normx<0$. Thus $f(-x)+\alpha \| -x\|<0$, i.e., $-x \in H$. To finish the proof it is sufficient to check that if $x \in \mbox{int}(S(f,\alpha))$, then $f(x)+\alpha \| x\| <0$.  To this end, consider $x_0 \in \mbox{int}(S(f,\alpha))$ and assume, contrary to our claim, that $f(x_0)+\alpha  \| x_0 \|= 0$. Choose $\epsilon>0$ such that $B(x_0,\epsilon):=\{x \in X\colon \| x-x_0\| \leq \epsilon\}\subset S(f,\alpha)$. Then  $f(x_0+y)+\alpha  \| x_0+y\|\leq 0$ for every $y \in \epsilon B_X$.  Now pick any $y \in \epsilon B_X$. It follows that $f(y)-\alpha \| y \|=f(y)-\alpha \| y \|+f(x_0)+\alpha  \| x_0 \|\leq f(x_0+y)+\alpha  \| x_0+y\|\leq 0$. Then $f(y)\leq \alpha \| y \|$. As $y$ was arbitrarily taken, the last inequality holds true for every $y \in \epsilon B_X$. Hence, $f(-y)\leq \alpha \| -y \|$ for every $y \in \epsilon B_X$. As a consequence, $|f(y)|\leq \alpha \| y \|$ for every $y \in \epsilon B_X$ and then $\| f \|_* \leq \alpha$, which contradicts the assumption $\alpha< \| f \|_*$.
\end{proof}

In the following result, we relax the assumption of a weak compact base in \cite[Theorem 3.7]{Kasimbeyli2010} and strengthen the conclusion of \cite[Corollary 3.3]{Kasimbeyli2010} by obtaining a cone with a bounded base instead of a pointed cone.
\begin{lemma}\label{lema:C_base_acotada_equiv_cono_dual_aumentado_no_vacio}
Let $X$ be a normed space and $\cC\subset X$ a convex cone. Then $\cC$ has a bounded base if and only if $\cC^{a\#}_+\not =\emptyset$.
\end{lemma}
\begin{proof}
$\Rightarrow$ Assume that $\cC$ has a bounded base, then $\cC$ is pointed. Now, by \cite[Theorem~1.1]{GARCIACASTANO20151178}, there exists $f \in S_{X^*}$ such that $\inf_{S_X\cap \cC}f=\delta>0$. Let us fix $x \in \cC\setminus \{\cero\}$. We have $f(\frac{x}{\normx})\geq \delta>0$. Hence $f \in \cC^{\#}$. On the other hand, we have $f(x)-\frac{\delta}{2}\normx>f(x)-\delta\normx \geq  0$. Thus $(f,\frac{\delta}{2})\in \cC^{a \#}_+$. $\Leftarrow$ Now, assume that $\cC^{a\#}_+\not =\emptyset$. Then, by Lemma \ref{lem:S(f,alpha)_base_acotada} assertion (i) we can pick some $(f,\alpha)\in \cC^{a*}_+$. Then $\inf_{S_X\cap\cC}f\geq \alpha>0$.  Applying again \cite[Theorem 1.1]{GARCIACASTANO20151178}, we conclude that $\cC$ has a bounded base. 
\end{proof}

\begin{proof}[Proof of Proposition \ref{prop:separacion_implica_base_acotada}]
By \cite[Theorem 3.1]{GARCIACASTANO2022}  we have $\cC^{a\#}_+\not =\emptyset$. Now, Lemma~\ref{lema:C_base_acotada_equiv_cono_dual_aumentado_no_vacio} applies and  $\cC$ has a bounded base.  
\end{proof}

Next, we define the two types of dilating cones that we will use in this work. The $\epsilon$-conic neighborhoods we introduce are the closure of the $\epsilon$-conic neighborhoods introduced in \cite[Definition 4.2]{Kasimbeyli2010}, while for the Henig dilating cones, we follow \cite[Definition 3.3]{Gopfert2004}.

We recall the definition of the distance between a point $x \in X$ and a subset $A \subset X$, denoted by $d(x,A)$, which is defined as $d(x,A) := \inf \{\| x - a\| \colon a \in A\}$. 

\begin{lemma}\label{Lemma_formula_de_la_clausura}
Let $X$ be a normed space and $A\subset X$ a subset. Then, the equality  $\overline{A + \epsilon B_X}=\{x\in X\colon d(x,A)\leq \epsilon\}$, holds for every $\epsilon>0$.
\end{lemma}
\begin{proof}
We will prove the equality by double inclusion. We first prove $\subset$. Fix any $x \in \overline{A + \epsilon B_X}$. Then, for any $n\geq 1$ there exists $x_n\in A+\epsilon B_X$ such that $\|x-x_n\|\leq \frac{1}{n}$. Fixed $n\geq 1$, we can write $x_n=a_n+\epsilon b_n$ for some $a_n\in A$ and $b_n\in B_X$. As a consequence, $\|x-a_n\|\leq \|x-x_n\|+\|\epsilon b_n\|\leq \frac{1}{n}+\epsilon$ and $\inf\{\|x-a\|\colon a \in A\}\leq \epsilon$. Now we prove the inclusion $\supset$. Let consider $x\in X$ such that $d(x,A)\leq \epsilon$. By definition of infimum, for every $n\geq 1$, there exists $y_n\in A$ such that $\|x-y_n\|\leq \epsilon+\frac{1}{n}$. We can write, for every $n\geq 1$, $x-y_n=(\epsilon+\frac{1}{n})b_n$ for some $b_n\in B_X$. As $\|x-y_n-\epsilon b_n\|\leq \frac{1}{n}$, it follows that $x=\lim_n(y_n+\epsilon b_n)\in \overline{A + \epsilon B_X}$.
\end{proof}

For a given cone $\cC \subset X$ and $0 < \epsilon < 1$, we define the set $S_{\epsilon} := \{x \in X \colon d(x, \cC \cap S_X) \leq \epsilon\}$. By the precedent lemma, we have $S_{\epsilon} = \overline{(S_X \cap \cC) + \epsilon B_X}$. Furthermore, if $\cC$ is convex and $B$ is a base of $\cC$, we define $B_{\epsilon} := \{x \in X \colon d(x,B) \leq \epsilon\}$ and $\delta_B := \inf_{b \in B} \| b \| > 0$. Additionally, we have the equality $B_{\epsilon} = \overline{B + \epsilon B_X}$.
\begin{definition}\label{def:epsilon_conic_neighbourhood} 
Let $X$ be a normed space and $\cC\subset X$ a cone.  We introduce the following definitions:
\begin{itemize}
\item[(i)]For every $0<\epsilon<1$, the cone $\cC_{\epsilon}:=\mbox{cone}(S_{\epsilon})$ is called a $\epsilon$-conic neighbourhood of $\cC$.  
\item[(ii)] Assuming that $\cC$ is convex and $B$ is a base of $\cC$, for every $0<\epsilon<\min\{1, \delta_B\}$, the cone $\cC_{(B,\epsilon)}:=\mbox{cone}(B_{\epsilon})$ is called a Henig dilating cone of $\cC$. 
\end{itemize}
\end{definition}
Below, we establish two lemmas that gather both the most elementary properties of the two introduced types of dilating cones and the way in which they can be ordered by inclusion. It should be noted that an $\epsilon$-conic neighborhood of a convex cone is not necessarily convex or pointed. To illustrate this, consider, for example, the cone $\cC:=\{(x,y)\in \mathbb{R}^2\colon x>0\}\cup\{(0,0)\}$ in $\mathbb{R}^2$ with the euclidean norm. Then, for every $0<\epsilon<1$, we have that the $\epsilon$-conic neighborhood $\cC_{\epsilon}=\{(x,y)\in \mathbb{R}^2\colon y \geq -\frac{\sqrt{1-\epsilon^2}}{\epsilon}x\}\cup \{(x,y)\in \mathbb{R}^2\colon y \leq \frac{\sqrt{1-\epsilon^2}}{\epsilon}x\}$  is a cone that is neither convex nor pointed. However, the following result provides certain properties of dilating cones that will be useful later.

\begin{lemma}\label{lema:propiedades_entornos_conicos}
Let $X$ be a normed space and $\cC\subset X$ a cone. The following statements hold.
\begin{itemize}
\item[(i)]  The cone $\cC_{\epsilon}$ is closed  for every $0<\epsilon<1$.
\item[(ii)] Assuming that $\cC$ is convex and $B$ is a base of $\cC$, there exists $0<\epsilon_B<1$ such that if $0<\epsilon <\epsilon_B$, then the cone $\cC_{\epsilon}$ is pointed and the cone $\cC_{(B, \epsilon)}$ is  closed and convex.  
\end{itemize}
\end{lemma}
\begin{proof}
(i) Fix  arbitrary $0<\epsilon<1$ and $x \in \overline{\cC_{\epsilon}}$. If $x=\cero$, then it is clear that $x \in \cC_{\epsilon}$. Assume that $x\not = \cero$,  we will check that $x \in \cC_{\epsilon}$ again.  Choose $m_0\geq 1$ such that $\frac{2}{m_0}<1-\epsilon$.  Consider $(x_n)_n \subset S_{\epsilon}$ and $(\lambda_n)_n \subset (0,+\infty)$ such that $x=\lim_n\lambda_nx_n$ and fix a sequence $(\epsilon_n)_n\subset (\epsilon,1)$ converging to $\epsilon$.  We claim that $S_{\epsilon}\subset \cC\cap S_X+\epsilon_n B_X$, for every $n\geq 1$.  Indeed, consider $z\in S_{\epsilon}=\overline{ \cC\cap S_X+\epsilon B_X}$. Then, there exists $z_n\in \cC\cap S_X+\epsilon B_X$ such that $\|z-z_n\|\leq \epsilon_n-\epsilon$ $\Rightarrow$ $z-z_n=(\epsilon_n-\epsilon)b_n$, for some $b_n\in B_X$.  On the other hand, for every $n\geq 1$, we can write $z_n=c_n+\epsilon b_n^*$ for some $c_n\in \cC\cap S_X$ and $b_n^*\in B_X$. Then, $z=c_n+\epsilon b_n^*+(\epsilon_n-\epsilon)b_n$ for every $n\geq 1$. As $\|\epsilon b_n^*+(\epsilon_n-\epsilon)b_n\|\leq \epsilon+\epsilon_n-\epsilon=\epsilon_n$, for every $n\geq 1$, it follows that $z \in  \cC\cap S_X+\epsilon_n B_X$, for every $n\geq 1$.  As a consequence,  $x_n \in S_{\epsilon}\subset \cC\cap S_X+\epsilon_n B_X$, for every $n\geq 1$. Then we choose sequences $(c_n)_n  \subset \cC\cap S_X$ and $(b_n)_n \subset B_X$ such that $x_n=c_n+\epsilon_nb_n\in S_{\epsilon}$ for every $n \geq 1$.  Thus $x=\lim_n\lambda_nx_n=\lim_n\lambda_n(c_n+\epsilon_nb_n)$. The sequence $(\lambda_n)_n$ is bounded above. Assume the contrary, then we would have $\cero =\lim_n\frac{x}{\lambda_n}=\lim_n(c_n+\epsilon_nb_n)$ (maybe for some subsequence), implying that $\cero \in \overline{S_{\epsilon}}=S_{\epsilon}$, which is not possible. 
Then,  it is not restrictive to assume that $(\lambda_n)_n$ converges to some $\lambda \in [0,+\infty)$. If $\lambda=0$, then $x=\lim_n\lambda_n(c_n+\epsilon_nb_n)=0$ because the sequence $(c_n+\epsilon_nb_n)_n$ is bounded, a contradiction.  Then $\lambda>0$ and $\|\frac{x}{\lambda}-c_n\|\leq\|\frac{x}{\lambda}-c_n-\epsilon_nb_n)\| +\|\epsilon_nb_n\|\leq\|\frac{x}{\lambda}-(c_n+\epsilon_nb_n)\| +\epsilon_n$ for every $n \geq 1$. This implies that $\frac{x}{\lambda}\in S_{\epsilon}$.  Indeed, fix $n_0>1$ and consider $\epsilon_{n_0}$ such that $\epsilon_{n_0}\leq \epsilon+\frac{1}{2n_0}$. Now, as $\frac{x}{\lambda}=\lim_n (c_n+\epsilon_nb_n)$, there exists $n_1\geq n_0$ such that $\|\frac{x}{\lambda}-(c_{n_1}+\epsilon_{n_1}b_{n_1})\|\leq \frac{1}{2n_0}$ and $\epsilon_{n_1}\leq \epsilon_{n_0}$. Then, $\|\frac{x}{\lambda}-c_{n_1}\|\leq\|\frac{x}{\lambda}-(c_{n_1}+\epsilon_{n_1}b_{n_1})\| +\epsilon_{n_1}\leq \frac{1}{2n_0}+\epsilon_{n_0}\leq  \frac{1}{2n_0}+\epsilon+\frac{1}{2n_0}=\epsilon+\frac{1}{n_0}$. Therefore, $\inf\{\|\frac{x}{\lambda}-y\|\colon y \in \cC\cap S_X\}\leq \epsilon$ implying that  $\frac{x}{\lambda}\in S_{\epsilon}$. Finally,  $x=\lambda \frac{x}{\lambda} \in \mbox{cone}(S_{\epsilon})=\cC_{\epsilon}$.

(ii) We first show that $\cC_{\epsilon}$ is pointed for $\epsilon>0$ small enough.  Assume that $\cC$ has a bounded base.  By \cite[Theorem 1.1]{GARCIACASTANO20151178}, there exists $f \in S_{X^*}$ such that $\inf_{\cC\cap S_X} f=\delta>0$. Fix $\epsilon_1>0$ such that $|f(x)|<\frac{\delta}{2}$ for every $x\in \epsilon_1 B_X$.  We pick $\epsilon<\epsilon_1$ and take an arbitrary $x\in \cC\cap S_X+\epsilon B_X$.  Write $x=x'+x''$ for $x'\in \cC\cap S_X$, $x''\in \epsilon B_X$.  Then $f(x)=f(x')+f(x'')>\delta-\frac{\delta}{2}=\frac{\delta}{2}$. Hence  $f(x)\geq \frac{\delta}{2}$ whenever $d(x, \cC\cap S_X)\leq \epsilon$, which yields $f \in (\cC_{\epsilon})^{\#}$ and the last implies that $\cC_{\epsilon}$ is pointed.  On the other hand,  in  \cite[Lemma~3.4]{Gopfert2004} it is proved that  $\cC_{(B,\epsilon)}$ is  closed and convex for every $0<\epsilon < \delta_B=\inf_{b\in B}  \{ \| b\|\}$. The proof is over after defining $\epsilon_B:= \min\{\delta_B, 1,\epsilon_1\}$.  
\end{proof}
\begin{lemma} \label{lema:dilatacion_dentro_entorno}
Let $X$ be a normed space, $\cC \subset X$ a convex cone having a bounded base,  and $0 < \epsilon <1$.  Then, there exist a bounded base $B$ of $\cC$ and $0<\epsilon' < \epsilon$ such that $\delta_B>\epsilon$,  $\cC_{(B,\epsilon)} \subset \cC_{\epsilon}$, and $\cC_{\alpha}\setminus \{\cero\} \subset int(\cC_{(B,\epsilon)})$ for every $0<\alpha<\epsilon'$.
\end{lemma}
\begin{proof}
Let $B'$ be an arbitrary bounded base of $\cC$ and take $m':=\inf_{b' \in B'}\| b' \|>0$. Then, the set $B:=\frac{1}{m'} B'$ is a bounded base of $\cC$ satisfying $\delta_B \geq 1$.  Let us prove the inclusion $\cC_{(B,\epsilon)} \subset \cC_{\epsilon}$.  Fix $x \in B+ \epsilon B_X$ and write $x=x_1+x_2$ for $x_1 \in B$ and $x_2 \in \epsilon B_X$. Then $\frac{x_1}{\| x_1\|} \in S_X \cap \cC $  and $\frac{x_2}{\| x_1\|} \in \frac{\epsilon}{\| x_1\|}B_X \subset \epsilon B_X$ as $\| x_1 \| \geq 1$. Consequently, $\frac{x}{\| x_1\|} \in S_X \cap \cC + \epsilon B_X$. Now fix an arbitrary $\lambda>0$.  Then $\lambda x=\| x_1\| \lambda(\frac{1}{\| x_1 \|}x_1+\frac{1}{\| x_1 \|}x_2)\in \mbox{cone}(S_X \cap \cC + \epsilon B_X)$. As a consequence,  $\mbox{cone}(B+ \epsilon B_X) \subset \mbox{cone}(S_X \cap \cC + \epsilon B_X)\subset \cC_{\epsilon}$.  By Lemma~\ref{lema:propiedades_entornos_conicos},  the cone $\cC_{\epsilon}$ is closed,  then $\overline{B+ \epsilon B_X} \subset \cC_{\epsilon}$. Hence
$\cC_{(B,\epsilon)}=\mbox{cone}(\overline{B+ \epsilon B_X})  \subset \cC_{\epsilon}$. Next, we will show the inclusion $\cC_{\alpha}\setminus \{\cero\} \subset \mbox{int}(\cC_{(B,\epsilon)})$ for $\alpha$ small enough. Let $0<M:=\sup_{b \in B } \| b \| < +\infty$ and take $\alpha \in (0, \frac{\epsilon}{2M})$.
Fix $x \in S_X \cap \cC + \alpha B_X$, write $x=x_1+x_2$ for $x_1 \in   S_X \cap \cC $, $x_2 \in \alpha B_X$, and choose $\lambda>0 $ such that $\lambda x_1 \in B$. Since $\| x_1 \| = 1$ and $\lambda x_1 \in B$,  we have $M \geq \| \lambda x_1 \|=  \lambda \| x_1 \|=\lambda$, and then $\lambda x=\lambda x_1+ \lambda x_2 \in B + \lambda \alpha B_X \subset B + \lambda \frac{\epsilon}{2M} B_X \subset B + M\frac{1}{2M} \epsilon B_X =B+ \frac{\epsilon}{2} B_X$. Consequently, $S_X \cap \cC + \alpha B_X \subset \mbox{cone}(B_{\frac{\epsilon}{2}})=\cC_{(B,\frac{\epsilon}{2})}$. By Lemma~\ref{lema:propiedades_entornos_conicos}, $\cC_{(B,\frac{\epsilon}{2})}$ is closed. Then $\overline{S_X \cap \cC + \alpha B_X} \subset \cC_{(B,\frac{\epsilon}{2})}$,  and so $\cC_{\alpha}=\mbox{cone}(\overline{S_X \cap \cC + \alpha B_X}) \subset \cC_{(B,\frac{\epsilon}{2})}$.  Now take $\epsilon':=\frac{\epsilon}{2M}$ and apply \cite[Lemma 2.1]{Gong1995}. Then we have the chain of inclusions $ \cC_{\alpha}\setminus\{\cero\}\subset \cC_{(B,\frac{\epsilon}{2})}\setminus\{\cero\} \subset \mbox{int}(\cC_{(B,\epsilon)})$ for every $\alpha \in (0, \epsilon').$  
\end{proof}

The next two theorems demonstrate how Bishop-Phelps cones can be interpolated between the ordering cone and the dilating cones under consideration. The first theorem improves upon \cite[Theorem 4.4]{Kasimbeyli2010} as it is stated for general normed spaces instead of reflexive Banach spaces. Moreover, it provides an equivalence for the SSP instead of just a necessary condition.
\begin{theorem}\label{thm:tma_separacion_epsilon_conic_neighbourhood}
Let $X$ be a normed space, $\cC\subset X$ a cone, and $0<\epsilon<1$. The following statements are equivalent.
\begin{itemize}
\item[(i)] $(\cC,\cC_{\epsilon})$ has the SSP.
\item[(ii)] There exist $\delta_2>\delta_1> 0$ and $f \in X^*$ such that $(f,\alpha)\in \cC^{a\#}_+$ and $-\cC\setminus \{\cero\}\subset \mbox{int}(S(f,\alpha))=\{x \in X\colon f(x)+\alpha \normx<0\}\subset -\cC_{\epsilon}$ satisfying  $f(x)+\alpha \normx> 0$ whenever $\alpha \in (\delta_1,\delta_2)$, $x \in X\setminus \mbox{int}(-\cC_{\epsilon})$, and $x \not = \cero$.
\end{itemize} 
\end{theorem}
\begin{proof}
(i)$\Rightarrow$(ii) Since $(\cC,\cC_{\epsilon})$ has the SSP, it follows that $(-\cC,-\cC_{\epsilon})$ has the SSP. By \cite[Theorem  3.1]{GARCIACASTANO2022}, there exist $\delta_2>\delta_1> 0$ and $f \in X^*$ such that for every $\alpha \in (\delta_1,\delta_2)$ we have $(f,\alpha)\in \cC^{a\#}_+$ satisfying $f(x)+\alpha \normx<0< f(y)+\alpha \| y \|$ whenever $x \in -\cC\setminus\{\cero\}$, $y \in X\setminus \mbox{int}(-\cC_{\epsilon})$, $y \not=\cero$. By Lemma \ref{lem:S(f,alpha)_base_acotada}, we have $\alpha <\| f \|_*$ and int$(S(f,\alpha))=\{x\in X \colon f(x)+\alpha \normx<0\}\not =\emptyset$. To finish the proof we will show that int$(S(f,\alpha))\subset -\cC_{\epsilon}$. Pick any $x\in \mbox{int}(S(f,\alpha))$ and assume, contrary to our claim, that $x \not \in -\cC_{\epsilon}$. Choose $z\in -\cC\subset -\cC_{\epsilon}$ and $\lambda \in (0,1)$ such that $y=\lambda x +(1-\lambda)z \in \mbox{bd}(-\cC_{\epsilon})$. Then $f(y)+\alpha \| y \| \leq \lambda(f(x)+\alpha \normx)+(1-\lambda)(f(z)+\alpha \| z \|)<0$, a contradiction.  

(ii)$\Rightarrow$(i)  A direct consequence of \cite[Theorem 3.1]{GARCIACASTANO2022}.
\end{proof}
Now we state the analogous result to the previous one, but for Henig dilating cones. The proof is a direct adaptation of the previous one, so it is omitted.
\begin{theorem}\label{thm:tma_separacion_Henig_dilating_cone}
Let $X$ be a normed space, $\cC\subset X$ a convex cone,  $B$ a bounded base of $\cC$,  and $0<\epsilon <\min\{1,\delta_B\}$. The following are equivalent.
\begin{itemize}
\item[(i)] $(\cC,\cC_{(B,\epsilon)})$ has the SSP.
\item[(ii)] There exist $0<\delta_1<\delta_2$ and $f \in X^*$ such that $(f,\alpha)\in \cC^{a\#}_+$ and $-\cC\setminus \{\cero\}\subset \mbox{int}(S(f,\alpha))=\{x \in X\colon f(x)+\alpha \normx<0\}\subset -\cC_{(B,\epsilon)}$ satisfying $f(x)+\alpha \normx> 0$ whenever  $\alpha \in (\delta_1,\delta_2)$,  $x \in X\setminus \mbox{int}(-\cC_{(B,\epsilon)})$, $x \not = \cero$.
\end{itemize} 
\end{theorem}
\begin{corollary} \label{coro:inclusiones_y_SSP}
Let $X$ be a normed space and $\cC$, $\cK_1, \cK_2$ cones such that $\cC\cap \cK_1\not =~\emptyset$.  If $(\cC,\cK_1)$ has the SSP and  $ \cK_1 \subset \cK_2$ , then $(\cC,\cK_2)$  has the SSP. 
\end{corollary}


\begin{proof}
Assume that $(\cC,\cK_1)$ has the SSP. By  \cite[Theorem 3.3]{GARCIACASTANO2022}, there exist $\delta_2>\delta_1> 0$ and $f \in X^*$ such that for every $\alpha \in (\delta_1,\delta_2)$ we have   $(f,\alpha) \in \cC^{a\#}_{+}$ and $-\cC\setminus \{\cero\}\subset \mbox{int}(S(f,\alpha))\subset -\cK_1$ satisfying $f(x)+\alpha \normx > 0$ for every $x \in X \setminus \mbox{int}(-\cK_1)$, $x \not = \cero$.  On the other hand, the inclusion $\cK_1 \subset \cK_2$ implies that $\mbox{bd}(-\cK_2)\subset   X\setminus \mbox{int}(-\cK_1)$. Now applying \cite[Theorem 3.1]{GARCIACASTANO2022}, we  have that $(\cC,\cK_2)$ has the SSP.
\end{proof}
In the following section, we will refer to the following observation on multiple occasions. This observation showcases the properties and interconnections of the two types of dilating cones when a given pair consisting of the ordering cone and an $\epsilon$-conic neighborhood satisfies the SSP. Therefore, it can be regarded as one of the key conclusions of this section. This observation is a consequence of  Proposition \ref{prop:separacion_implica_base_acotada}, Lemmas~\ref{lema:propiedades_entornos_conicos} and \ref{lema:dilatacion_dentro_entorno}, and Corollary \ref{coro:inclusiones_y_SSP}. 
\begin{remark}\label{rem:resumen_existencia_conos}
Let $X$ be  a normed space,  $\cC\subset X$ a convex cone, and  $0<\epsilon<1$.  If $(\cC,\cC_{\epsilon})$ has the SSP, then there exist a bounded base $B$ of $\cC$ and   $0<\epsilon' <\epsilon$ such that $\epsilon<\delta_B$, and the following properties hold:
\begin{itemize}
\item[(i)] $\cC_{\alpha}\setminus \{\cero\} \subset \mbox{int}(\cC_{(B,\epsilon)}) \subset \cC_{(B,\epsilon)} \subset \cC_{\epsilon}$ for every $0<\alpha <\epsilon'$.
\item[(ii)] The cone $\cC_{(B,\epsilon)}$  is closed and convex.
\item[(iii)] If  $(\cC,\cC_{\beta})$ satisfies the SSP for some $0<\beta < \epsilon'$, then $(\cC,\cC_{(B,\epsilon)})$ has the SSP.
\end{itemize}
\end{remark}

\section{Existence and density results}\label{sec:exist_henig}
In this section, we establish the main results of the paper. We present two types of results: first, existence results for Henig global proper efficient points, specifically sufficient conditions for their existence; second, we provide sufficient conditions for the density of these proper efficient points in the efficient set. The hypotheses in these results mainly consist of two parts: the SSP between the family of pairs given by the ordering cone and its conic neighborhoods, and some condition of weak compactness for the set under consideration. In the first theorem, it is not necessary to assume the convexity or boundedness of the involved set, and in its proof, we will use the following  fact.
\begin{fact}\label{fact_subconjunto_compacto}
Let $X$ be a Hausdorff topological space and $A$, $B$, $C\subset X$ subsets such that $C$ is closed and $C\subset B$.  If $A\cap B$ is compact, then so is $A\cap C$.
\end{fact}

\begin{theorem}\label{thm:Existence_He_Hartley_compact_nuevo}
Let $X$ be a partially ordered normed space, $\cC$ the ordering cone, and $A\subset X$ a subset. Assume that $(\cC,\cC_{\varepsilon})$ has the SSP for every $0< \varepsilon<1$.  If there exist $x_0\in A$ and $0<\delta<1$ such that the section $(x_0-\cC_{\delta})\cap A$ is  weak compact, then 
$(x_0-\cC_{\delta})\cap A\cap\mbox{GHe}(A,\cC)\not = \emptyset$.
\end{theorem}
\begin{proof}
We fix  $x_0\in A$ and $0<\delta<1$ from the statement. By Remark \ref{rem:resumen_existencia_conos}, there exists a bounded base $B$ of $\cC$ such that $\cC_{(B,\delta)} \subset \cC_{\delta}$, the cone $\cC_{(B,\delta)}$ is closed and convex, and $(\cC,\cC_{(B,\delta)})$ has the SSP. By convexity,  the cone $\cC_{(B,\delta)}$  is also weak closed. On the other hand,  the set $(x_0-\cC_{\delta})\cap A$ is  weak compact and hence, by Fact \ref{fact_subconjunto_compacto}, so is $(x_0-\cC_{(B,\delta)})\cap A$. Since $(\cC,\cC_{(B,\delta)})$ has the SSP,  Theorem~\ref{thm:tma_separacion_Henig_dilating_cone} applies and there exists $(f,\alpha)\in \cC^{a\#}_+$ such that $-\cC\setminus \{\cero\}\subset \mbox{int}(S(f,\alpha))=\{x \in X\colon f(x)+\alpha \normx< 0\}\subset -\cC_{(B,\delta)}$.  Now, we define the function  $g_{f,\alpha}:X\rightarrow\R$ by $g_{f,\alpha}(x):=f(x)+\alpha\| x \|$ for every $x \in X$.  As $\| \cdot \|$ is weak lower semicontinuous, so is $g_{f,\alpha}$.  By \cite[Theorem 2.43]{ali-bor},  the function $g_{f,\alpha}$ restricted to $(x_0-\cC_{(B,\delta)})\cap A$ attains its minimum at some $x_1 \in (x_0-\cC_{(B,\delta)})\cap A$. 
Hence for every $x \in (x_0-\cC_{(B,\delta)})\cap A$, $x \not = x_1$, we have $f(x)+\alpha \normx\geq f(x_1)+\alpha \| x_1\|$. This implies $0\leq f(x-x_1)+\alpha[\normx-\| x_1\|]\leq f(x-x_1)+\alpha \| x-x_1\|$. Denote $ I(f, \alpha)=\mbox{int}(S(f,\alpha))\cup \{\cero \}$. Then $x-x_1\not \in I(f,\alpha)$ for every $x \in (x_0-\cC_{(B,\delta)})\cap A$ such that $x\not = x_1$, yielding
\begin{equation}\label{eq:1_prueba_thm_Hartley_compact_nuevo}
(x_1+I(f, \alpha))\cap(x_0-\cC_{(B,\delta)})\cap A=\{x_1\}.
\end{equation} 
Then $x_1 \in \mbox{Min}((x_0-\cC_{(B,\delta)})\cap A,-I(f, \alpha))$. Now, as $-\cC\setminus \{\cero\}\subset \mbox{int}(S(f,\alpha))$ and $I(f, \alpha)$ is a convex cone, we get that $x_1 \in \mbox{GHe}((x_0-\cC_{(B,\delta)})\cap A,\cC)$.  Next, we will prove the equality $(x_1+I(f, \alpha))\cap A=\{x_1\}$ which leads to $x_1 \in \mbox{GHe}(A,\cC)$. Indeed,  $x_1 \in x_0-\cC_{(B,\delta)}$, $I(f, \alpha) \subset S(f,\alpha)\subset -\cC_{(B,\delta)}$ and the latter is a convex cone, so $-\cC_{(B,\delta)}-\cC_{(B,\delta)}\subset -\cC_{(B,\delta)}$. Hence, $(x_1+I(f, \alpha))\cap A \subset (x_0-\cC_{(B,\delta)}+I(f, \alpha))\cap A\subset (x_0-\cC_{(B,\delta)})\cap A$.  Finally,  by (\ref{eq:1_prueba_thm_Hartley_compact_nuevo}), we get $(x_1+I(f, \alpha)) \cap A\subset (x_1+I(f, \alpha))\cap(x_0-\cC_{(B,\delta)})\cap A=\{x_1\}$.
\end{proof}

The following example illustrates the above theorem.

\begin{example} \label{ej_existencia_con_debil_compacidad}
Let us consider $(\mathbb{R}^2, \| \, \|_2)$,  the cone $\cC:=\{(x,y)\in \mathbb{R}^2: y\geq | x |\}$ ordering the space, and the set
$A=\{(x,y)\in \mathbb{R}^2: -\frac{\pi}{2}\leq x \leq \frac{\pi}{2}, y\geq \sin(-| x |) \}$. Then, $(\cC,\cC_{\epsilon})$ has the SSP for every $0<\epsilon<1$, the set $((0,0)-\cC_{\frac{\sqrt{2}}{2}})\cap A$ is weak compact, and 
\begin{equation} \label{eq1}
 \text{GHe}(A,\cC)=\{(x,y)\in \mathbb{R}^2: -\frac{\pi}{2}\leq x \leq \frac{\pi}{2},  y=\sin(-| x |) \} \setminus \{ (0,0)\}.
\end{equation}

\end{example}

\begin{proof}
First, let us check that $(\cC, \cC_{\varepsilon})$ has the SSP for every $0<\varepsilon<1$. Notice that $\cC$ is a Bishop-Phelps cone.  Indeed, taking $\bar{f}:=(0,1)$ and  $\bar{\alpha}=\frac{\sqrt{2}}{2}$, we have that $\cC=C(\bar{f}, \bar{\alpha})$. Taking into account the particular form of the cone $\cC$,  it is easy to check that  for every $0<\varepsilon<1$, there exists $\rho>0$ such that $C(\bar{f}, \bar{\alpha}-\rho)\subset C_{\varepsilon}$. Now,  Example~\ref{EjempoBPtieneSSP} and  Corollary \ref{coro:inclusiones_y_SSP} yield that $(\cC, \cC_{\varepsilon})$ has the SSP for every $0<\varepsilon<1.$  

To prove the second claim, we take  $x_0=(0,0)$ and $\delta=\frac{\sqrt{2}}{2}$. Then $\cC_{\delta}=\mathbb{R} \times \mathbb{R}_+$ and 
$(x_0-\cC_{\delta})\cap A= \{(x,y)\in \mathbb{R}^2: -\frac{\pi}{2}\leq x \leq \frac{\pi}{2},0 \geq y\geq \sin(-| x |) \} $,
which is a compact (and therefore also a weak compact) set. 
Under these conditions, Theorem \ref{thm:Existence_He_Hartley_compact_nuevo} guarantees the existence of a Henig global proper efficient point of $A$ in $(x_0-\cC_{\delta})\cap A$. 

Finally, Equality (\ref{eq1}) follows from the fact that taking into account the slope of the curve $y=\sin(-| x |)$ at any $x \in [-\frac{\pi}{2}, \frac{\pi}{2}]\setminus \{0\} $, we can estimate
for every $x \in [-\frac{\pi}{2}, \frac{\pi}{2}]\setminus \{0\} $ a small enough $\delta_x>0$  such that 
$((x,\sin(-| x |)-\cC_{\delta_x})\cap A=\{ (x,\sin(-| x |)\}.$
At the origin $(0,0)$ the situation is different. We have that $((0,0)-\cC)\cap A  = \{ (0,0)\}$ and then $(0,0) \in \mbox{Min}(A)$, but we also have that $((0,0)-\cC_{\delta })\cap A \not = \{ (x,\sin(-| x |)\}$, for any  $\delta>0$. 
This happens because the slope of $y=\sin(-|x|)$ at $x=0$ is just $1$  on the left and $-1$ on the right, with the curve $y=\sin(-|x|)$ lying above  the half-lines defined by the slopes, which in turn define the boundary of $-\cC$.
Consequently, 
\begin{equation} \label{eq2}
 \text{Min}(A,\cC)=\{(x,y)\in \mathbb{R}^2: -\frac{\pi}{2}\leq x \leq \frac{\pi}{2},  y=\sin(-| x |) \}, 
\end{equation}  
since no point above the curve $y=\sin(-|x|)$ can be a minimal point.

\end{proof}

\begin{remark}
The previous theorem extends the sufficient condition given in \cite[Theorem~6]{Kasimbeyli2015} by removing the requirement of reflexivity for the space and relaxing the assumption of Hartley cone-weak compactness. In fact, our result only requires the weak compactness of a single section.
\end{remark}
Now we present the version of the previous theorem for Henig dilating cones. This result will be applied in the proof of the following theorem to make use of the convexity of the Henig dilating cones. The proof is omitted as it is an adaptation of the proof of the previous theorem.
\begin{theorem}\label{thm:Existence_GHe_con_Henig_dilating_cones_no_reflexivos}
Let $X$ be a partially ordered normed space, $\cC$ the ordering cone, $B$ a bounded base of $\cC$ such that $(\cC,\cC_{(B,\delta)})$ has the SSP for some $0<\delta<1$, and $A\subset X$.  If there exist $x_0\in A$ such that the section $(x_0-\cC_{(B,\delta)})\cap A$ is weak compact, then  $(x_0-\cC_{(B,\delta)})\cap A\cap \mbox{GHe}(A,\cC)\not = \emptyset$.
\end{theorem}

The following theorem is one of our main results. It assumes weak compactness of the considered set. This type of condition is common in existence results. Moreover, it is worth noting that our theorem does not require convexity conditions.
\begin{theorem}\label{thm:Existence_He_Cjto_compactbis}
Let $X$ be a partially ordered normed space, $\cC$ the ordering cone, and $A\subset X$ a subset. Assume that $(\cC,\cC_{\alpha})$ has the SSP for every $0<\alpha<1$. If $A$ is  weak compact,  then $\mbox{GHe}(A,\cC)\not = \emptyset$.
\end{theorem}
\begin{proof}
Fix $0<\epsilon<1$ and $x_0\in A$. By Remark \ref{rem:resumen_existencia_conos}, there exists a bounded base $B$ such that $\cC_{(B,\epsilon)}\subset \cC_{\epsilon}$ and $(\cC,\cC_{(B,\epsilon)})$ has the SSP. Since $\cC_{(B,\epsilon)}$ is weak closed, it follows that $A\cap (x_0-\cC_{(B,\epsilon)})$ is weak compact. Now, Theorem \ref{thm:Existence_GHe_con_Henig_dilating_cones_no_reflexivos} applies and the proof is over.
\end{proof}
We begin the final part of this section with two technical results that determine sufficient conditions for the sections given by the dilating cones we are dealing with to be included in subsets of balls as small as desired. These results are key for the density results and the Arrow, Barankin, Blackwell-type theorems that we will obtain at the end of this section. In the proof of the subsequent result we will denote by weak-$\lim_nx_n$ the limit of the sequence $(x_n)_n$ under the weak topology.

%

\begin{proposition}\label{prop:arrow_barankin_blackwell_1}
Let $X$ be a partially ordered normed space, $\cC$ the ordering cone,  and $A\subset X$. Assume that $\cC$ is a closed cone having a bounded base 
and there exists $0<\delta<1$ such that the section $A \cap \left( -\cC_{\delta} \right) $ is weak compact.  If $\cero \in \mbox{Min}(A,\cC)$, then for every $0<\epsilon <1$ there exists  $n_{\epsilon} \in \mathbb{N}$ such that
$A \cap ( -\cC_{\frac{1}{n_{\epsilon}}}) \subset A \cap (\epsilon B_X)$.
\end{proposition}
\begin{proof}
Fix $0<\delta<1$ from the statement and assume, contrary to our claim, that there exists $0<\epsilon<1$ such that for every $n \in \mathbb{N}$ we can pick $x_n \in A \cap \left( -\cC_{\frac{1}{n}} \right) $ satisfying  $\| x_n \| > \epsilon$.  We will check that under the former assumption we could  find a cluster point $x_0\not =\cero$  of the former sequence $(x_n)_{n\in \mathbb{N}}$ under the weak topology such that $x_0\in A \cap ( -\cC)$,  then contradicting the assumption $\cero \in \mbox{Min}(A,\cC)$. For this purpose, we will find a strictly increasing sequence $(n_k)_{k\geq 1}\subset \mathbb{N}$ and a sequence $(B_k)_{k\geq 1}$ of bounded bases of $\cC$ such that $\cC \subset  \cC_{\frac{1}{n_{k+1}}} \subset \cC_{(B_{k},\frac{1}{n_k})}\subset \cC_{\frac{1}{n_k}} \subset \cC_{\delta}$ for every $k\geq 1$.  We now proceed by induction.  Fix some $n_1 \in \mathbb{N}$ such that  $0<\frac{1}{n_1}<\delta$ and take  $\cC_{\frac{1}{n_1}} \subset \cC_{\delta}$.   We now apply Lemma \ref{lema:dilatacion_dentro_entorno}  to $\cC$ and $\epsilon=\frac{1}{n_1}$, then there exist a bounded base $B_1$ of $\cC$ and  $0<\epsilon'<\epsilon=\frac{1}{n_1}$ such that $\delta_{B_1}>\frac{1}{n_1}$, $\cC_{(B_{1},\frac{1}{n_1})}\subset \cC_{\frac{1}{n_1}}$, and $\cC_{\alpha}\setminus \{0_X\} \subset \mbox{int}(\cC_{(B_{1},\frac{1}{n_1})})$ for every $0<\alpha<\epsilon'$. Then, we pick $n_2\in \mathbb{N}$ such that $\frac{1}{n_2}<\epsilon'$.  As a consequence, $n_2 > n_1$ and $\cC_{\frac{1}{n_2}} \subset\cC_{(B_1,\frac{1}{n_1})} \subset \cC_{\frac{1}{n_1}}\subset \cC_{\delta}$. 
Now assume that for some $k_0>1$ we have $\cC \subset  \cC_{\frac{1}{n_{k_0+1}}} \subset \cC_{(B_{k_0},\frac{1}{n_{k_0}})}\subset \cC_{\frac{1}{n_{k_0}}} \subset \cC_{\delta}$.  By Lemma \ref{lema:dilatacion_dentro_entorno} applied to $\cC$ and $\epsilon=\frac{1}{n_{k_0+1}}$,  there exist a bounded base $B_{k_0+1}$ of $\cC$ and  $n_{k_0+2} > n_{k_0+1}$ such that  $\cC \subset\cC_{\frac{1}{n_{k_0+2}}} \subset\cC_{(B_{k_0+1},\frac{1}{n_{k_0+1}})}\subset \cC_{\frac{1}{n_{k_0+1}}}\subset \cC_{\delta}$ and the induction finishes. 
%
%
%
%
%
Now,   note that $x_{n_k} \in  A\cap (-\cC_{\frac{1}{n_k}})$ for every $k\geq~1$,  then $(x_{n_k})_{k\geq j+1} \subset A\cap (-\cC_{(B_j,\frac{1}{n_j})}) \subset A\cap (-\cC_{\delta} )$ for every $j\geq 1$. On the other hand,  each cone  $\cC_{(B_j,\frac{1}{n_j})}$ is closed and  convex, and then weakly closed for every $j\geq 1$.  Then the set $A\cap (-\cC_{(B_j,\frac{1}{n_j})})$ is  weakly compact for every $j\geq 1$ because $A\cap (-\cC_{(B_j,\frac{1}{n_j})})\subset A \cap (-\cC_{\delta})$ and the latter  is a weakly compact set. As a consequence,  the sequence $(x_{n_k})_{k\geq 2} \subset A \cap (-\cC_{(B_1,\frac{1}{n_1})})$ has a subsequence, namely $(x_{n_{k_\ell}})_{\ell\geq 1}$, that converges to some $x_0\in A \cap (-\cC_{(B_1,\frac{1}{n_1})})$ under the weak topology.  The condition  $\| x_n \| > \epsilon$ for every $n\in \mathbb{N}$ yields $x_0\not =\cero$. But  $x_0\in A\cap (-\cC_{(B_j,\frac{1}{n_j})})$ for every $j\geq 1$, because $(x_{n_k})_{k\geq j+1} \subset A\cap (-\cC_{(B_j,\frac{1}{n_j})})$.  Indeed,  the sequence $(x_{n_k})_{k\geq 2}\subset A\cap (-\cC_{(B_1,\frac{1}{n_1})})$ has the subsequence $(x_{n_{k_\ell}})_{\ell\geq 1}\subset A\cap (-\cC_{(B_1,\frac{1}{n_1})})$ that satisfies weak-$\lim_{\ell}x_{n_{k_\ell}}=x_0 \in A\cap (-\cC_{(B_1,\frac{1}{n_1})})$.  Now, since  $(x_{n_{k_\ell}})_{\ell\geq 2}$ is a subsequence of $(x_{n_k})_{k\geq 3}\subset A\cap (-\cC_{(B_2,\frac{1}{n_2})})$, it follows that $(x_{n_{k_\ell}})_{\ell\geq 2}\subset A\cap (-\cC_{(B_2,\frac{1}{n_2})})$ and weak-$\lim_{\ell}x_{n_{k_\ell}}=x_0 \in A\cap (-\cC_{(B_2,\frac{1}{n_2})})$ because the set $A\cap (-\cC_{(B_2,\frac{1}{n_2})})$ is weakly closed. In this way, for every $m>  1$ we have that $(x_{n_{k_\ell}})_{\ell\geq m}$ is a subsequence of $(x_{n_k})_{k\geq m+1}\subset A\cap (-\cC_{(B_m,\frac{1}{n_m})})$, it follows that $(x_{n_{k_\ell}})_{\ell\geq m}\subset A\cap (-\cC_{(B_m,\frac{1}{n_m})})$ and weak-$\lim_{\ell}x_{n_{k_\ell}}=x_0 \in A\cap (-\cC_{(B_m,\frac{1}{n_m})})$ because the set $A\cap (-\cC_{(B_m,\frac{1}{n_m})})$ is weakly closed. Then $x_0\in \cap_{j=1}^{+\infty}(-\cC_{(B_j,\frac{1}{n_j})})=\cap_{j=2}^{+\infty}(-\cC_{\frac{1}{n_j}})$. Now, we will check that $x_0\in -\overline{\cC}$ and the closedness of $\cC$ will imply that $x_0\in  A\cap (-\cC)$, which is impossible. To prove $x_0\in -\overline{\cC}$, we will pick an arbitrary $\eta>0$ and we will check that $B(x_0;\eta)\cap ( -\cC)\not =\emptyset$.   Since $x_0\in \cap_{j\geq 2}(-\cC_{\frac{1}{n_j}})$ we can write,  for every $j \geq 2$, $x_0=-\lambda_jz_j$ for some $\lambda_j>0$ and $z_j\in X$ such that $d(z_j,\cC\cap S_X)\leq \frac{1}{n_j}$. Since $\cC\cap S_X$ is bounded and $\cero \not \in \overline{\cC\cap S_X}$, there exist $0<M_1<M_2$ such that $M_1\leq \| z_j\|\leq M_2$ for every $j \geq 1$. Indeed, fix an arbitrary $j_0\geq 1$ and let $x_{j_0}\in \cC\cap S_X$ such that $\|z_{j_0}-x_{j_0}\|< \frac{3}{2n_{j_0}}$. Then $\|z_{j_0}\|\leq \|z_{j_0}-x_{j_0}\|+\|x_{j_0}\|<\frac{3}{2n_{j_0}}+1\leq \frac{3}{2}+1=\frac{5}{2}=:M_2$.  To find $M_1$ we restrict our argument to $j_0\geq 2$. In this case we have that $n_{j_0}\geq 2$ and we have $\|z_{j_0}\|\geq \|x_{j_0}\|-\|z_{j_0}-x_{j_0}\|=1-\|z_{j_0}-x_{j_0}\|>1-\frac{3}{2n_{j_0}}\geq 1-\frac{3}{4}=\frac{1}{4}=:M_1^*$.  Then, we take $M_1:=\min\{M_1^*,\|z_1\|\}$. Therefore, $\lambda_j=\frac{\| x_0\|}{\| z_j\|}\leq \frac{\| x_0\|}{M_1}$, i.e., the sequence $\{\lambda_j\}_j\subset [0,+\infty)$ is bounded.  It is not restrictive to assume that such a sequence is, in fact, convergent to some $\lambda_0 \in [0,+\infty)$. It is clear that $\lambda_0>0$ because $\frac{\| x_0\|}{\lambda_j}= \| z_j\|\leq M_2$ for every $j \geq 1$. Indeed,  $\lambda_0=0$ would imply $\lim_j \| z_j\|=+\infty$. Fix $i \geq 1$ such that $\frac{2\lambda_{i}}{n_i}<\eta$. Since $d(z_{i},\cC\cap S_X)\leq \frac{1}{n_i}$, there exists $c \in \cC$ such that $\| z_{i}-\frac{c}{\| c \|}\| \leq \frac{2}{n_i}$. Define $u:=-\frac{c}{\|c\|}\lambda_i\in -\cC$, then $u \in B(x_0;\eta)\cap (-\cC)$ because $\|x_0-u\|=\lambda_i\|-z_i+\frac{c}{\|c\|}\|<\frac{2\lambda_{i}}{n_i}<\eta$ and the proof is over.
\end{proof}

\begin{proposition}\label{prop:arrow_barankin_blackwell_Henig_dilating_cones}
Let $X$ be a partially ordered normed space, $\cC$ the ordering cone,  and $A\subset X$ a weak compact subset. Assume that $\cC$ is a closed cone satisfying that $(\cC,\cC_{\alpha})$ has the SSP for every $0<\alpha<1$. If $\cero \in \mbox{Min}(A,\cC)$, then for every $0<\epsilon <1$ there exist  $n_{\epsilon} \in \mathbb{N}$ and a bounded base $B(\epsilon)$ of $\cC$ such that $\frac{1}{n_{\epsilon}}<\delta_{B(\epsilon)}$, $(\cC,\cC_{(B(\epsilon),\frac{1}{n_{\epsilon}})})$ has the SSP,  and $A \cap ( -\cC_{(B(\epsilon),\frac{1}{n_{\epsilon}})}) \subset A \cap (\epsilon B_X)$.
\end{proposition}
\begin{proof}
Fix arbitrary $0<\epsilon<1$. We will show that there exists  $n_{\epsilon} \in \mathbb{N}$ such that $A \cap( -\cC_{\frac{1}{n_{\epsilon}}}) \subset A \cap (\epsilon B_X)$.  Let us assume the contrary, i.e.,  assume that there exists $0<\epsilon<1$ such that for every $n \in \mathbb{N}$ we can pick $x_n \in A \cap (-\cC_{\frac{1}{n}}) $ satisfying  $\| x_n \| > \epsilon$.  Repeating the argument used in the proof of Proposition \ref{prop:arrow_barankin_blackwell_1} (applying Rermark \ref{rem:resumen_existencia_conos} instead of Lemma \ref{lema:dilatacion_dentro_entorno}) we have a strictly increasing sequence $(n_k)_{k\geq 1}\subset \mathbb{N}$ and a sequence $(B_k)_{k\geq 1}$ of bounded bases of $\cC$ such that the pair $(\cC,\cC_{(B_{k},\frac{1}{n_k})})$ has the SSP and $\cC \subset  \cC_{\frac{1}{n_{k+1}}} \subset \cC_{(B_{k},\frac{1}{n_k})}\subset \cC_{\frac{1}{n_k}} \subset \cC_{\delta}$ for every $k\geq 1$.   Each cone  $\cC_{(B_j,\frac{1}{n_j})}$ is closed and  convex, and then weakly closed for every $j\geq 1$.  Then the set $A\cap (-\cC_{(B_j,\frac{1}{n_j})})$ is  weakly compact for every $j\geq 1$ because $A\cap (-\cC_{(B_j,\frac{1}{n_j})})\subset A$ and $A$ is a weak compact set.  As a consequence,  for every $j\geq 1$ the sequence $(x_{n_k})_{k\geq j+1} \subset A\cap (-\cC_{(B_j,\frac{1}{n_j})})$  has a subsequence that converges to some $x_0\in A \cap (-\cC_{(B_j,\frac{1}{n_j})})$, $x_0\not = \cero$, under the weak topology.   Then  $x_0\in \cap_{j=1}^{+\infty}(-\cC_{(B_j,\frac{1}{n_j})})=\cap_{j=2}^{+\infty}(-\cC_{\frac{1}{n_j}})$, as in the final part of the proof of Proposition \ref{prop:arrow_barankin_blackwell_1}. It follows that $x_0\in -\overline{\cC}$  and hence $x_0\in  A\cap (-\cC)$ which is impossible.  Now we consider $n_{\epsilon} \in \mathbb{N}$ such that $A \cap \left( -\cC_{\frac{1}{n_{\epsilon}}}\right) \subset A \cap (\epsilon B_X) $. 
By Remark \ref{rem:resumen_existencia_conos}, there exists a bounded base $B(\epsilon)$ of $\cC$ such that $(\cC,\cC_{(B(\epsilon),\frac{1}{n_{\epsilon}})})$ has the SSP and $\cC_{(B(\epsilon),\frac{1}{n_{\epsilon}})} \subset \cC_{\frac{1}{n_{\epsilon}}}$.  Then we have $A \cap ( -\cC_{(B(\epsilon),\frac{1}{n_{\epsilon}})}) \subset A \cap (\epsilon B_X)$ and the proof is over.
\end{proof}

Our first density result is a local approximation theorem. In the proof of this theorem, the convexity of the Henig dilating cones plays a crucial role.
\begin{theorem}\label{thm:arrow_barankin_blackwell_local_approximation}
Let $X$ be a partially ordered normed space, $\cC$ the ordering cone,  $A\subset X$, and $\bar{x}\in A$.  Assume that $\cC$ is a closed cone satisfying that $(\cC,\cC_\alpha)$ has the SSP for every $0<\alpha<1$ and there exists $0<\delta<1$ such that the section $A \cap \left( \bar{x}-\cC_{\delta} \right) $ is weak compact. If $\bar{x} \in \mbox{Min}(A,\cC)$, then for every $\epsilon>0$ there exists $x_{\epsilon}\in \mbox{GHe}(A,\cC)$ such that $\| \bar{x}-x_{\epsilon}\|<\epsilon$.
\end{theorem}
 
\begin{proof}
Fix $\epsilon>0$. We first consider the case $\bar{x}  =\cero$.  As ordering cones are assumed to be convex,  by Proposition \ref{prop:separacion_implica_base_acotada}, it follows that $\cC$ has a bounded base. Then, we apply Proposition~\ref{prop:arrow_barankin_blackwell_1}  and there exists  $n_{\epsilon} \in \mathbb{N}$ such that $A \cap (-\cC_{\frac{1}{n_{\epsilon}}}) \subset A \cap (\epsilon B_X) $.
Now, take  $m> \mbox{max}\{ n_{\epsilon}, \frac{1}{\delta} \}$ and apply Remark \ref{rem:resumen_existencia_conos} to $\cC$ and $\epsilon=\frac{1}{m}$. Then, there exists $B$ a bounded base of $\cC$ such that $\cC_{(B,\frac{1}{m})}\subset \cC_{\frac{1}{m}}$ and the pair of cones $(\cC,\cC_{(B,\frac{1}{m})})$ has the SSP.  Since $\cC_{(B,\frac{1}{m})}$ is weak closed and $A\cap (-\cC_{(B,\frac{1}{m})})\subset A\cap (-\cC_{\delta})$, it follows that $A\cap (-\cC_{(B,\frac{1}{m})})$ is weak compact.  Now, Theorem \ref{thm:Existence_GHe_con_Henig_dilating_cones_no_reflexivos}  applies and there exists $x_{\epsilon}\in A\cap (-\cC_{(B,\frac{1}{m})})$ such that $x_{\epsilon}\in \mbox{GHe}(A,\cC)$. Finally, the inclusion $A\cap (-\cC_{(B,\frac{1}{m})})\subset  A \cap (\epsilon B_X)$ implies $\|x_{\epsilon}\|<\epsilon$. Now consider the case $\bar{x}\not =\cero$. Then $\cero\in\mbox{Min}(A-\bar{x},\cC)$ and $(A-\bar{x})\cap (-\cC_{\delta})$ is weak compact. Then, from the above, there exists $y_{\epsilon}\in\mbox{GHe}(A-\bar{x},\cC)$ such that $\| y_{\epsilon}\|<\epsilon$. Clearly, $x_{\epsilon}:=\bar{x}+y_{\epsilon} \in\mbox{GHe}(A,\cC)$ and $\| \bar{x}-x_{\epsilon}\|<\epsilon$. 
\end{proof}

The following example illustrates the above theorem.
\begin{example} \label{ej_existencia_densidad}
Consider again  Example \ref{ej_existencia_con_debil_compacidad}, let us take $\bar{x}=(0,0)\in \mbox{Min}(A,\cC)$  and $\delta=\frac{\sqrt{2}}{2}$. Then, the conditions of Theorem \ref{thm:arrow_barankin_blackwell_local_approximation} hold, and we have guaranteed that for every $\epsilon>0$ there exists $x_{\epsilon}\in \mbox{GHe}(A,\cC)$ such that $\| \bar{x}-x_{\epsilon}\|<\epsilon$.
Indeed, this is what happens in the example as 
$\underset{x\rightarrow 0}{\lim} \sin(-| x |)=0$
according to (\ref{eq1}) and (\ref{eq2}).
\end{example}

The previous theorem yields the following global density theorem.
\begin{corollary}\label{coro:arrow_barankin_blackwell1}
Let $X$ be a partially ordered normed space, $\cC$ the ordering cone, and $A\subset X$.  Assume that $\cC$ is a closed cone satisfying that $(\cC,\cC_\alpha)$ has the SSP for every $0<\alpha<1$. If for every $x\in \mbox{Min}(A,\cC)$ there exists  $0<\delta_x<1$ such that $A \cap (x-\cC_{\delta_x}) $ is weak compact, then $\mbox{Min}(A,\cC)\subset \overline{\mbox{GHe}(A,\cC)}$.
\end{corollary}

\begin{remark}
The previous corollary extends \cite[Theorem~3.2]{Kasimbeyli2016} by avoiding the requirement of space reflexive, considering a broader set of proper efficient points, and weakening the assumption of norm compactness for each section of the set by only requiring weak compactness for one section of each minimal point.
\end{remark}

Next, we present our final density theorem. It is worth noting that the hypothesis of weak compactness of the set implies its boundedness; however, the convexity assumption is still not necessary.
\begin{theorem}\label{thm:arrow_barankin_blackwell_A_weak_compact}
Let $X$ be a partially ordered normed space, $\cC$ the ordering cone, and $A\subset X$.  Assume that $\cC$ is a closed cone satisfying that $(\cC,\cC_\alpha)$ has the SSP for every $0<\alpha<1$. If $A$ is weak compact, then $\mbox{Min}(A,\cC)\subset \overline{\mbox{GHe}(A,\cC)}$.
\end{theorem}

\begin{proof}
Fix $\epsilon>0$. We first consider the case $\cero \in \mbox{Min}(A,\cC)$. We will check that there exists $x_{\epsilon}\in {\mbox{GHe}(A,\cC)}$ such that $\|x_{\epsilon}\|<\epsilon$. By Proposition \ref{prop:arrow_barankin_blackwell_Henig_dilating_cones}, there exist $n_{\epsilon} \in \mathbb{N}$ and a bounded base $B_{\epsilon}$ of $\cC$ such that $(\cC,\cC_{(B_{\epsilon},\frac{1}{n_{\epsilon}})})$ has the SSP and $A \cap ( -\cC_{(B_{\epsilon},\frac{1}{n_{\epsilon}})}) \subset A \cap (\epsilon B_X)$. Since $\cC_{(B_{\epsilon},\frac{1}{n_{\epsilon}})}$ is weak closed and $A$ weak compact, it follows that  $A \cap ( -\cC_{(B_{\epsilon},\frac{1}{n_{\epsilon}})})$ is weak compact. By Theorem \ref{thm:Existence_GHe_con_Henig_dilating_cones_no_reflexivos}, there exists $x_{\epsilon}\in A \cap ( -\cC_{(B_{\epsilon},\frac{1}{n_{\epsilon}})})$ such that $x_{\epsilon}\in {\mbox{GHe}(A,\cC)}$. Then  $x_{\epsilon}\in A \cap (\epsilon B_X)$ implying $\|x_{\epsilon}\|<\epsilon$.  Clearly $\cero \in \overline{\mbox{GHe}(A,\cC)}$. To finish the proof assume that $x \in \mbox{Min}(A,\cC)$ and $x\not =\cero$.  Then $\cero\in\mbox{Min}(A-x,\cC)$ and $A-x$ is weak compact. Then, from the above, there exists $y_{\epsilon}\in\mbox{GHe}(A-x,\cC)$ such that $\| y_{\epsilon}\|<\epsilon$. Clearly, $x_{\epsilon}:=x+y_{\epsilon} \in\mbox{GHe}(A,\cC)$ and $\| x-x_{\epsilon}\|<\epsilon$.  So $x \in \overline{\mbox{GHe}(A,\cC)}$. 
\end{proof}

The following example illustrates the above theorem.
\begin{example} \label{ej_existencia_densidad}
Consider again Example \ref{ej_existencia_con_debil_compacidad}, we will trim the set $A$ there to make it compact (and then weak compact) here. Then, we  define the set
$A'=A \cap \{ y\leq 1\} =\{(x,y)\in \mathbb{R}^2: -\frac{\pi}{2}\leq x \leq \frac{\pi}{2},1\geq  y\geq \sin(-| x |) \},$. Then, for $A'$ the conditions of Theorem \ref{thm:arrow_barankin_blackwell_A_weak_compact} hold and then we have guaranteed that  $\mbox{Min}(A,\cC)\subset \overline{\mbox{GHe}(A,\cC)}$. And this is what happens in the example where as $\text{Min}(A,\cC)=\text{Min}(A',\cC)$ and  $\text{GHe}(A,\cC)=\text{GHe}(A',\cC)$, from (\ref{eq1}) and (\ref{eq2}) we directly obtain that
$\mbox{Min}(A',\cC)\subset \overline{\mbox{GHe}(A',\cC)}$.

\end{example}


We conclude our work with a question that sets a direction for future research:\begin{problem}
Can the SSP condition be relaxed  in the statement of Theorems \ref{thm:Existence_He_Hartley_compact_nuevo}, \ref{thm:arrow_barankin_blackwell_local_approximation}, or \ref{thm:arrow_barankin_blackwell_A_weak_compact}?
\end{problem}
This question suggests exploring whether the SSP condition can be weakened while still obtaining similar results and conclusions. By relaxing this condition, it may be possible to extend the applicability of the theorems to a broader class of cones and sets. Further investigation in this direction could lead to new insights and generalizations in the field.
\section*{Acknowledgement(s)}
We would like to express our gratitude to the referee for their valuable suggestions, which have significantly improved this work.
\section*{Funding}
Fernando Garc\'ia-Casta\~no and M. A. Melguizo-Padial acknowledge the financial support from the Spanish Ministry of Science, Innovation and Universities (MCIN/AEI) under project PID2021-122126NB-C32,  funded by MICIU/AEI /10.13039/501100011033/ and FEDER A way of making Europe.




  \bibliographystyle{abbrvnat}
  \bibliography{references}
\end{document}